\documentclass[12pt]{article}

\usepackage{amsmath,amsthm}
\usepackage{amssymb}
\usepackage{color}
\usepackage{breqn}
\usepackage{enumerate}
\usepackage{graphicx}
\usepackage{bbm}
\usepackage[affil-it]{authblk}
\usepackage{tabu}
\usepackage{bold-extra}
\usepackage[hypertex]{hyperref}

\usepackage{authblk}
\usepackage[center]{titlesec}
\usepackage{lipsum}

\newtheorem{theorem}{Theorem}
\newtheorem{corollary}[theorem]{Corollary}
\newtheorem{lemma}[theorem]{Lemma}
\newtheorem{proposition}[theorem]{Proposition}

\theoremstyle{definition}

\newcommand{\tightoverset}[2]{
  \mathop{#2}\limits^{\vbox to -.5ex{\kern-1.15ex\hbox{$#1$}\vss}}}
\newcommand{\conv}{
	\tightoverset{\boldsymbol{-}}{\ast}
}

\newcommand{\pol}{
	\operatorname{pol}
}

\newcommand{\supp}{
	\operatorname{supp}
}			

\newcommand\restr[2]{{
  \left.\kern-\nulldelimiterspace 
  #1 
  \vphantom{\big|} 
  \right|_{#2} 
}}
	
\newcounter{boldSectionCounter}
\stepcounter{boldSectionCounter}
\newcounter{boldSubsectionCounter}
\stepcounter{boldSubsectionCounter}

\newcommand{\codim}{\operatorname{codim}}

\newcommand{\boldSection}[1]{
   \large\begin{center}\noindent {\bfseries{\scshape \S \arabic{boldSectionCounter} #1}}\\[6pt]\end{center}\normalsize
   \stepcounter{boldSectionCounter}
	 \setcounter{boldSubsectionCounter}{1}
}

\makeatletter
\def\blfootnote{\gdef\@thefnmark{}\@footnotetext}
\makeatother	

\title{A bilinear version of Bogolyubov's theorem}
\date{}
\author[1]{W. T. Gowers}
\affil[1]{\footnotesize Royal Society 2010 Anniversary Research Professor, University of Cambridge}
\author{L. Mili\'cevi\'c}
\affil[2]{\footnotesize Mathematical Institute of the Serbian Academy of Sciences and Arts}
\begin{document}
\maketitle




\begin{abstract} A theorem of Bogolyubov states that for every dense set $A$ in $\mathbb{Z}_N$ we may find a large Bohr set inside $A+A-A-A$. In this note, motivated by the work on a quantitative inverse theorem for the Gowers $U^4$ norm, we prove a bilinear variant of this result in vector spaces over finite fields. Namely, if we start with a dense set $A \subset \mathbb{F}^n_p \times \mathbb{F}^n_p$ and then take rows (respectively columns) of $A$ and change each row (respectively column) to the set difference of it with itself, repeating this procedure several times, we obtain a bilinear analogue of a Bohr set inside the resulting set, namely the zero set of a biaffine map from $\mathbb{F}^n_p \times \mathbb{F}^n_p$ to a $\mathbb{F}_p$-vector space of bounded dimension. An almost identical result was proved independently by Bienvenu and L\^e.
\end{abstract}


\section{Introduction}

An important theorem of Bogolyubov~\cite{BogPaper} states that whenever $A$ is a dense subset of $\mathbb{Z}_N$, then $A+A-A-A$ contains a large Bohr set. This argument has played a crucial role in many important results in additive combinatorics. For example, in his groundbreaking proof of Freiman's theorem~\cite{Ruzsa}, Ruzsa makes a clever use of this argument, and in the proof of Szemer\'edi's theorem by the first author~\cite{GowersSzemeredi} the argument appears in several forms.\\

The main question that we consider in this note is how to generalize the Bogolyubov theorem to the bilinear setting. In this note, we focus on the case where the ambient group is $\mathbb{F}_p^n$.\\

Our motivation for this work was to obtain a tool that could be used to prove a quantitative inverse theorem for Gowers $U^4$ norms over finite fields. If we look at the proof of Szemer\'edi's theorem for arithmetic progressions of length 4 in~\cite{GowersSzemeredi}, or the proof of Green and Tao's quantitative inverse theorem for $U^3$ norms over finite fields~\cite{GTU3}, we see that Freiman's theorem, and Bogolyubov's argument in particular, play an important role. Thus, in order to come up with a formulation of Bogolyubov's argument in the bilinear setting, it is natural to examine the proof in~\cite{GowersSzemeredi} for arithmetic progressions of length 5. It turns out that the starting point is the following. If $A$ is a dense set of the ambient group $G$, without arithmetic progressions of length 5, then it has large $U^4$ norm. This in turn, after some algebraic manipulations, naturally gives a dense set $X \subset G \times G$ and a map $\phi \colon X \to G$ (which is defined by looking at large Fourier coefficients of some functions related to the indicator function for $A$), with the property that $\phi$ respects many pairs of vertical parallelograms of same width and height\footnote{A \emph{vertical parallelogram} of width $w$ and height $h$ is any quadruple of points of the form $(x,y), (x,y + h), (x + w, z), (x + w, z + h)$. We say that $\phi$ \emph{respects} a pair of vertical parallelograms $\Big((x,y), (x,y + h), (x + w, z), (x + w, z + h)\Big)$ and $\Big((x',y'), (x',y' + h), (x' + w, z'), (x' + w, z' + h')\Big)$ if $\phi(x,y) - \phi(x,y + h) - \phi (x + w, z) + \phi (x + w, z + h) = \phi(x',y') - \phi(x',y' + h) - \phi (x' + w, z') + \phi (x' + w, z' + h)$}. Once again, going back to the $U^3$ norm case, in that setting we get a map that has many additive quadruples, which immediately leads us to Freiman's theorem. Thus, vertical parallelograms should point us the way towards bilinear variants of Freiman's theorem and Bogolyubov argument.\\  

Given a set $B \subset G$, we may first define the set $B^1$ of all $(x, h)$ such that in the line $\{x\} \times G$ there is a vertical segment of height $h$ between points in $B$, i.e. there is $y \in G$ such that $(x,y), (x, y + h) \in B$. Let $B^2$ be the set of all $(w,h)$ such that there is a parallelogram of width $w$ and height $h$ in $B$. Observe crucially that $B^2$ is the set of all $(w, h)$ such that in the line $G \times \{h\}$ there is a horizontal segment of length $w$ between points in $B^1$. Thus, we obtain the set of all vertical parallelogram heights and widths by convolving $B$ with itself in the $y$-direction first, followed by a convolution in the $x$-direction, and finally taking the support of the resulting function. This motivates the idea that the bilinear version of Bogolyubov argument should involve convolving the starting set with itself a few times like this, each time in a fixed direction.\\

It remains to decide what kind of structure we should aim for inside the support of the convolved set. In the original Bogolyubov argument, we obtain a Bohr set, which is a set defined by solutions to several inequalities of the form $|\alpha(x)-1|\leq\epsilon$, where $\alpha$ is a character. In the finite-fields setting of a finite fields, this becomes solving $\alpha(x) = 0$ for several linear functionals $\alpha$. Thus, a natural candidate for the structured set would be the set of solutions to several equations of the form $\beta(x,y) = 0$, where each $\beta \colon G \times G \to \mathbb{F}_p$ is a bilinear map.\\


Let us note here that we also prove somewhat different versions of the bilinear Bogolyubov argument in the forthcoming papers~\cite{GM1},~\cite{GM2}. Our reason for writing a note on this result is that it might be of separate interest and it is not exactly what we need in the proof of the inverse theorem.\\

\section{Bilinear Bohr varieties and statement of results}

To state the theorem, we must first give a few definitions. Let $G$ and $W$ be $\mathbb{F}_p$-vector spaces\footnote{All vector spaces in this note are finite-dimensional.}. A subset $B \subset G \times G$ is called a \emph{bilinear Bohr variety} if there is a biaffine map (affine in each variable) $\beta \colon G\times G \to W$ such that $B = \{(x,y) \colon \beta(x,y) = 0\}$. If $k = \dim W$, then we say that $B$ has \emph{codimension} at most $k$ and write $\codim B \leq k$.\\

We convolve $A \subset G \times G$ with itself as follows. We first convolve once in $y$ direction, thus defining
\[A^{(1)} = \cup_{x \in G}\hspace{2pt} \{x\} \times (A_{x \cdot} - A_{x \cdot}).\]
Next, we convolve twice in $x$-direction, setting
\[A^{(2)} = \cup_{y \in G}\hspace{2pt} (A^{(1)}_{\cdot y} + A^{(1)}_{\cdot y} - A^{(1)}_{\cdot y} - A^{(1)}_{\cdot y}) \times \{y\}.\]
Finally, we convolve twice in $y$-direction
\[A^{(3)} = \cup_{x \in G}\hspace{2pt} \{x\} \times (A^{(2)}_{x \cdot} + A^{(2)}_{x \cdot} - A^{(2)}_{x \cdot} - A^{(2)}_{x \cdot}).\]

\begin{theorem}\label{mainTheorem} Let $A \subset G \times G$ be a set of density $c$. Define the $A^{(3)}$ as above. Then $A^{(3)}$ contains $(U \times V) \cap \mathbf{B^L}$ where $U, V$ are subspaces and $\mathbf{B^L}$ is a bilinear Bohr variety defined by bilinear maps, such that
\[\codim U, \codim V, \codim \mathbf{B^L} \leq p^{\exp (\log^{O(1)} c^{-1})}.\]
\end{theorem}

We remark that very similar results were proved independently by Bienvenu and L\^e in~\cite{BienvenuLe}.\\

This theorem has the following rather pleasant corollary.

\begin{corollary}\label{bisubspace} Let $A \subset G \times G$ be any set of density $c$ such that, $A_{x \cdot}$ and $A_{\cdot y}$ are subspaces for all $x, y$ (possibly empty sets). Then $A$ contains $(U \times V) \cap \mathbf{B^L}$ where $U, V$ are subspaces and $\mathbf{B^L}$ is a bilinear Bohr variety defined by bilinear maps, such that
\[\codim U, \codim V, \codim \mathbf{B^L} \leq p^{\exp (\log^{O(1)} c^{-1}))}.\]
\end{corollary}

\iftrue
\else

We also apply Theorem~\ref{mainTheorem} to deduce a classification theorem for trilinear forms of low rank. For $\mathbb{F}_p$-vector spaces $V, W$, we denote the space of all linear maps from $V$ to $W$ by $\mathcal{L}(V, W)$.

\begin{theorem}[Structure of low-rank trilinear forms.]\label{trilinearTheorem} Let $U, V, W$ be vector spaces over $\mathbb{F}_p$, and let $T \colon U \to \mathcal{L}(V, W)$ be a linear map with the property that for at least $c |U|$ of $u \in U$, we have $\operatorname{rank} T(u) \leq k$. Then, we may find subspaces $U' \leq U,$ $V' \leq V$ and $E \leq W$ such that, if we write $\xi = \log^{O(1)} c^{-1} k^{O(1)} \log^{O(1)} p$
\[\codim_U U', \codim_V V' \leq \exp(\exp(\xi)),\hspace{2pt}\dim E \leq \exp(\exp(\exp(\xi))),\text{ and}\]
\[T(u)(v) \in E.\]\end{theorem}

In order to deduce this theorem, we prove another result that might be of independent interest, which we record here as well.

\begin{proposition}[Ramsey-like theorem for bilinear maps.]\label{dichotomy} For any $a,b \in \mathbb{N}$, there is $K$ (we may take $K = (2p)^{\frac{1}{2}(a+1)^2 b}$) such that the following holds. If we are given $\mathbb{F}_p$-vector spaces $X, Y, Z$ and a bilinear map $\beta \colon X \times Y \to Z$ such that $\operatorname{im} \beta$ contains at least $K$ linearly independent vectors, then there is $x \in X$ such that $\operatorname{rank} \beta_{x \cdot} \geq a$ or there is $y \in Y$ such that $\operatorname{rank} \beta_{\cdot y} \geq b$.\end{proposition}

\fi

It would be interesting to find the correct bounds, or at least better ones, for Theorem~\ref{mainTheorem}. Also, given the additional algebraic structure, it is possible that Corollary~\ref{bisubspace} can be proved more directly, but at the time of writing we do not see how to do this. \\


\noindent\textbf{Notation.} We will write $G$ for the ambient group, which is currently $G = \mathbb{F}_p^n$, and $N$ for the size of $G$. For a subset $S \subset G \times G$, we write $S_{x \cdot} = \{y \in G\colon (x,y) \in S\}$ and $S_{\cdot y} = \{x \in G\colon (x,y) \in S\}$. Similarly, if $f\colon A \subset G \times G \to X$ is a function, for $x \in G$, we define function $f_{x \cdot} \colon A_{x \cdot} \to X$ by $f_{x \cdot}(y) = f(x,y)$, and we analogously define $f_{\cdot y} \colon A_{\cdot y} \to X$.\\
\indent As is customary, Fourier coefficients will be defined using expectations. Thus, for $f \colon G \to \mathbb{C}$, we define $\hat{f}(r) = \mathbb{E}_{x \in G} f(x) \omega^{-r \cdot x} = \frac{1}{N} \sum_{x \in G} f(x) \omega^{-r \cdot x}$. Also, we use a slightly non-standard convolution and we write $f \conv g(x) = \mathbb{E}_y f(y+x)\overline{g(y)}$.\\
\indent Write $f \leq \pol(a_1, \dots, a_k)$ or $f = \pol(a_1, \dots, a_k)$, if $f$ is a function of $a_1, \dots, a_k$, and there is a polynomial $p$ in $a_1, \dots, a_k$, with positive coefficients only, such that $f = O(p)$. We also allow more complex, but natural, notation such as $\exp(-\pol(x)) \leq f \leq \exp(\pol(x))$ (this means that $|\log f| \leq \pol(x)$). This is less standard substitute for the equivalent $X = O(Y^{O(1)})$ notation, with aim to make the bounds in proofs more readable.\\



\noindent\textbf{Acknowledgements.} The second author would like to thank Trinity College and the Department of Pure Mathematics and Mathematical Statistics at the University of Cambridge for their generous support while this work was carried out. He also acknowledges the support of the Ministry of Education, Science and Technological Development of the Republic of Serbia, Grant III044006.


\section{Proof of the main theorem}

Before we embark on the proof, we need to recall a couple of lemmas. As one might expect, the first one is the usual Bogolyubov argument.

\begin{lemma}\label{bogArg} Let $f \colon G \to [0,1]$ be given. Then $\operatorname{supp} (f \conv f) \conv (f \conv f)$ contains $S^\perp$, where $S = \Big\{r \colon \Big|\widehat{f}(r)\Big| \geq \Big|\widehat{f}(0)\Big|^{\frac{3}{2}}\Big\}$ and $|S| \leq \Big| \widehat{f}(0)\Big|^{-2}$. \end{lemma}

\begin{proof}We have $\widehat{f \conv f}(r) = \Big|\widehat{f}(r)\Big|^2$ and $\widehat{(f \conv f) \conv (f \conv f)}(r) = \Big|\widehat{f}(r)\Big|^4$. Thus,
\[\begin{split}(f \conv f) \conv (f \conv f)(x) = \sum_{r} \widehat{(f \conv f) \conv (f \conv f)}(r) \omega^{rx} = 
\sum_r \Big|\widehat{f}(r)\Big|^4 \omega^{rx} = \sum_{r \in S} \Big|\widehat{f}(r)\Big|^4 \omega^{rx} + \sum_{r \notin S} \Big|\widehat{f}(r)\Big|^4 \omega^{rx}.\end{split}\]
When $x \in S^\perp$, we have
\[\sum_{r \in S} \Big|\widehat{f}(r)\Big|^4 \omega^{rx} = \sum_{r \in S} \Big|\widehat{f}(r)\Big|^4 \geq \Big|\widehat{f}(0)\Big|^4.\]
On the other hand, 
\[\bigg|  \sum_{r \notin S} \Big|\widehat{f}(r)\Big|^4 \omega^{rx} \bigg| \leq  \sum_{r \notin S} \Big|\widehat{f}(r)\Big|^4 < \Big|\widehat{f}(0)\Big|^3  \sum_{r} \Big|\widehat{f}(r)\Big|^2 = \Big|\widehat{f}(0)\Big|^3  \mathbb{E}_y \Big|f(y)\Big|^2 \leq \Big|\widehat{f}(0)\Big|^3  \mathbb{E}_y \Big|f(y)\Big| = \Big|\widehat{f}(0)\Big|^4,\]
thus $(f \conv f) \conv (f \conv f)(x) \not= 0$, and hence $S^\perp \subset \supp (f \conv f) \conv (f \conv f)$. Finally, 
\[\Big|\widehat{f}(0)\Big|^3 \Big|S\Big| \leq \sum_r \Big|\widehat{f}(r)\Big|^2 = \mathbb{E}_x |f(x)|^2 \leq \mathbb{E}_x |f(x)| = \Big| \widehat{f}(0)\Big|,\]
implying that $|S| \leq \Big| \widehat{f}(0)\Big|^{-2}$.\end{proof}

The second ingredient we recall is the following lemma that appears implicitly in~\cite{GowersSzemeredi}. We actually couple the work in Lemma 13.1 and Corollary 13.2 with Sanders's bounds in Freiman's theorem~\cite{Sanders}. (Note that the proofs of Lemma 13.1 and Corollary 13.2 stay the same when translated to $\mathbb{F}^n_p$.)


\begin{corollary}\label{FourierLinearity2} Let $f \colon G\times G \to \mathbb{C}$ be any function with $\|f\|_\infty \leq 1$. For any $x, y$, define a function
\[g(x, y) = \Big(f_{x\cdot} \conv f_{x\cdot}\Big) (y).\] 
Given $\xi$ and a set $B \subset G$ with density $c$, there is $k = \exp(\pol(\xi^{-1}, c^{-1}))$ and there are affine functions $\alpha_1, \alpha_2, \dots, \alpha_k\colon G \to G$ and a subset $B' \subset B$ such that 
\[|B'| \geq (1-\xi) |B|,\]
and such that if $y \in B'$ and $\Big|\widehat{g_{\cdot y}}(r)\Big| \geq \xi$, then
\[r \in \Big\{\alpha_1(y), \dots, \alpha_k(y)\Big\}.\]
\end{corollary}

\begin{proof}We note that Lemma 13.1 of~\cite{GowersSzemeredi} holds with exactly the same proof for $G = \mathbb{F}_p^n$ instead $\mathbb{Z}_N$. To prove this corollary, we iteratively pick large Fourier coefficients of $\widehat{g_{\cdot y}}$. Lemma 13.1 then tells us that any such choice, viewed as a map from $G$ to $G$, has many additive quadruples. To relate such a map to an affine map, we use the following corollary of the Balog-Szemer\'edi-Gowers theorem~\cite{BalogSzemeredi,GowersSzemeredi} and Sanders's bounds for Freiman's theorem~\cite{Sanders}.

\begin{proposition}\label{bsg} Let $A \subset G$ and $\phi \colon A \to G$ by any map with at least $c N^3$ additive quadruples. Then, there is an affine map $\alpha \colon G \to G$ such that $\alpha(x) = \phi(x)$ holds for at least $\exp(-\pol( \log(c^{-1}))) N$ of values $x \in G$.\end{proposition}

We proceed as follows. Begin by taking all $y \in B$ and for each $y$, list its large Fourier coefficients, i.e. all $r$ such that $\Big|\widehat{g_{\cdot y}}(r)\Big| \geq \xi$. By Parseval's identity the number of such $r$ is at most $\xi^{-2}$, since $\|g_{\cdot y}\|_\infty \leq 1$. At each step, we take some these coefficients, show that they are actually given by a common affine map, and remove them from the list. The procedure terminates when the number of $y$ such that at least one of its large Fourier coefficients remains in the list becomes not greater than $\xi |B|$.\\
\indent Suppose that we still have more than $\xi |B|$ of $y$ that have at least one Fourier coefficient in the list. Let $A$ be the set of such $y$, and define $\sigma \colon A \to G$ so that for every $y \in A$, $\sigma(y)$ is still in the list. Then
\[\sum_{y \in A}\Big|\widehat{g_{\cdot y}}(\sigma(y))\Big|^2 \geq \xi^3 N,\]
so Lemma 13.1 of~\cite{GowersSzemeredi} tells us that $\sigma$ has at least $\xi^{12}$ additive quadruples. Proposition~\ref{bsg} implies that there is an affine map $\alpha\colon G \to G$ that coincides in the value with $\sigma$ in at least $\exp(-\pol( \log(\xi^{-1}))) N$ elements of $G$. Remove these elements from the list and repeat the argument.\\
\indent The procedure therefore terminates after $k$ steps, where $k \leq \exp(\pol( \log(\xi^{-1})))$, giving affine maps $\alpha_1, \dots, \alpha_k \colon G \to G$. Let $B' \subset B$ be the set of $y \in B$ whose all Fourier coefficients were removed from the list. Thus, if $y \in B'$ and $r$ satisfies $\Big|\widehat{g_{\cdot y}}(r)\Big| \geq \xi$, then we have $r \in \{\alpha_1(y), \dots, \alpha_k(y)\}$, as claimed.\end{proof}

\noindent\textbf{Bilinear Bohr varieties generate structure.} We begin our work by proving that we may use a bilinear Bohr variety $\mathbf{B}$ to generate structure in the following sense. Imagine we pick some rows in $G \times G$, (which we view as the input) and then select those columns that have a dense intersection with the chosen rows and $\mathbf{B}$ (which we view as output). Then the set of coordinates of these columns contains a dense 
subspace. Thus, given any input, we obtain a very structured output.

\begin{proposition}\label{subspaceDenseContainment} Let $S \subset G$ have size $cN$ and let $\mathbf{B}$ be a bilinear Bohr variety defined by biaffine maps $\beta_1, \dots, \beta_k \colon G \times G \to \mathbb{F}_p$ of the form $\beta_i(x,y) = x \cdot \alpha_i(y)$ for an affine map $\alpha_i \colon G \to G$. Let $\mathbf{r}(x)$ be the number of $y \in S$ such that $x \in \mathbf{B}_{\cdot y}$, and let $X = \Big\{x \in G\colon \mathbf{r}(x) \geq \frac{c}{2p^k}N\Big\}$. Then $X$ contains a subspace $V$ of codimension at most $p^k(k + 4c^{-1}p^{2k})$. \end{proposition}

\begin{proof} Let first $V_0 \subset G = \{\alpha_1(0), \dots, \alpha_k(0)\}^{\bot}$, a subspace of codimension at most $k$. Observe that for $x \in V_0$, we have $(x,y) \in \mathbf{B}$ if and only if $x \cdot \alpha^L_i(y) = 0$ holds for all $i = 1,\dots, k$. Thus, for $x \in V_0$,  
\begin{equation*}\begin{split}\mathbf{r}(x) &= \sum_{y \in S} \mathbb{E}_{\lambda_1, \dots, \lambda_k \in \mathbb{F}_p}\omega^{-\sum_{i=1}^k \lambda_i x \cdot \alpha^L_i(y)} = \sum_{y \in S} \mathbb{E}_{\lambda_1, \dots, \lambda_k \in \mathbb{F}_p} \omega^{-y \cdot (\sum_{i=1}^k \lambda_i \alpha_i^{L T}(x))}\\ &=\mathbb{E}_{\lambda_1, \dots, \lambda_k \in \mathbb{F}_p} N \cdot \Big(\mathbb{E}_y S(y) \omega^{-y \cdot (\sum_{i=1}^k \lambda_i \alpha_i^{L T}(x))}\Big)\\
&=N\cdot \mathbb{E}_{\lambda_1, \dots, \lambda_k \in \mathbb{F}_p} \widehat{S}\bigg(\sum_{i=1}^k \lambda_i \alpha_i^{L T}(x)\bigg),  \end{split}\end{equation*}
where $\alpha^{LT}$ stands for the transpose of the linearization of an affine map $\alpha$.\\

Let $R$ be the large spectrum of $S$, defined by $R = \Big\{r\in G\colon |\widehat{S}(r)| \geq \frac{c}{2p^k}\Big\}$. Then $\frac{c^2}{4p^{2k}}|R| \leq \sum_r \Big|\widehat{S}(r)\Big|^2 = \mathbb{E}_y \Big|S(y)\Big|^2 = c$, so $|R|\leq 4c^{-1}p^{2k}$. Take any maximal subspace $V_1 \subset V_0$ such that $R \cap V_1 = \{0\}$, so $\operatorname{codim} V_1 \leq k + 4c^{-1}p^{2k}$. Finally, let
\[V_2 = \cap_{\lambda_1, \dots, \lambda_k \in \mathbb{F}_p} \bigg(\sum_{i=1}^k\lambda_i \alpha_i^{LT}\bigg)^{-1}\Big(V_1\Big),\]
which has codimension at most $p^k(k + 4c^{-1}p^{2k})$. We prove that this is the desired subspace.\\

Let $x \in V_2$. By definition, for any choice of $\lambda_1, \dots, \lambda_k$, we have $\sum_{i=1}^k\lambda_i \alpha_i^{LT}(x) \in V_1$. Thus, if $\sum_{i=1}^k\lambda_i \alpha_i^{LT}(x) \in R$, then $\sum_{i=1}^k\lambda_i \alpha_i^{LT}(x) = 0$. Hence,
\[\text{either  }\Big|\widehat{S}\Big(\sum_{i=1}^k\lambda_i \alpha_i^{LT}(x)\Big)\Big| < \frac{c}{2p^k}\text{  or  }\sum_{i=1}^k\lambda_i \alpha_i^{LT}(x) = 0.\]
Therefore, let $s$ be the number of choices of $\lambda_1, \dots, \lambda_k \in \mathbb{F}_p$ such that $\sum_{i=1}^k\lambda_i \alpha_i^{LT}(x) = 0$. In particular, $s\geq 1$, as we can set $\lambda_1 = \dots = \lambda_k = 0$. From this we deduce that when $x \in V_2$,
\[\Big|\frac{p^k}{N}\mathbf{r}(x) - s\hspace{2pt}\widehat{S}(0)\Big| \leq p^k \cdot\frac{c}{2p^k} = \frac{c}{2},\]
so $\mathbf{r}(x) \geq \frac{c}{2p^k}N$.
\end{proof} 

\begin{proof}[Proof of Theorem~\ref{mainTheorem}.] Let $A \subset G \times G$ be an arbitrary set of density $c$. First, convolve once in the $y$-direction, setting 
\[f(x,y) = A_{x \cdot } \conv A_{x \cdot} (y).\]
Thus, the set $A^{(1)}$ is exactly the support of $f$. In particular, we have
\[\mathbb{E}_{x,y} f(x,y) = \mathbb{E}_{x,y}A_{x \cdot } \conv A_{x \cdot} (y) \geq N^{-3}\sum_x \Big|A_{x \cdot}\Big|^2 \geq N^{-4} \Big(\sum_x |A_{x \cdot}|\Big)^2 \geq c^2.\]
Next, convolve twice in the $x$-direction, and define a new map $g \colon G \times G \to \mathbb{C}$ by
\[g(x,y)  = \Big(f_{\cdot y} \conv f_{\cdot y}\Big) \conv \Big(f_{\cdot y} \conv f_{\cdot y}\Big)(x).\]
This time, $A^{(2)} = \operatorname{supp} g$.\\[3pt]
\noindent\textbf{First application of Bogolyubov's argument.} By Bogolyubov's argument (Lemma~\ref{bogArg}), each row of $A^{(2)}$ contains a subspace given by the orthogonal complement of the large spectrum $S_y$ of $f_{\cdot y}$, given by $S_y = \Big\{r \colon \Big|\widehat{f_{\cdot y}}\Big| \geq \Big|\widehat{f_{\cdot y}}(0)\Big|^{\frac{3}{2}}\Big\}$, and we have $\Big|S_y\Big| \leq \Big|\widehat{f_{\cdot y}}(0)\Big|^{-2}$, and $\Big|\widehat{f_{\cdot y}}(0)\Big| = \mathbb{E}_x f(x,y)$. Since $\mathbb{E}_{x,y} f(x,y) \geq c^2$, by averaging, we have a set $Y \subset G$ of size at least $\frac{c^2}{2} N$ such that $\mathbb{E}_x f(x,y) \geq \frac{c^2}{2}$ for all $y \in Y$.\\

\noindent\textbf{Applying Corollary~\ref{FourierLinearity2}.} Next, for any fixed $\xi > 0$, by Corollary~\ref{FourierLinearity2}, we obtain $k \leq \exp(\pol(\log \xi^{-1}))$, a subset $Y' \subset Y$ of size at least $|Y'| \geq (1-\xi)|Y|$ and affine maps $\alpha_1, \dots, \alpha_k$ such that for $y \in Y'$, the Fourier coefficients $\Big|\widehat{f_{\cdot y}}(r)\Big| \geq \xi$ obey $r \in \{\alpha_1(y), \dots, \alpha_k(y)\}$. We take $\xi$ such that for all $y \in Y$, $\xi \leq \Big|\widehat{f_{\cdot y}}(0)\Big|^{\frac{3}{2}}$. But, for $y \in Y$, we have $\Big|\widehat{f_{\cdot y}}(0)\Big| \geq \frac{c^2}{2}$, so we may take $\xi = \frac{c^3}{4}$, making $k \leq \exp(\pol( \log c^{-1}))$.\\ 
\indent Let $\mathbf{B}$ be the bilinear Bohr variety defined by the $k$ biaffine maps $(x,y) \mapsto x \cdot \alpha_i(y)$. Hence, $A^{(2)}$ contains a set $C$ of the form $C = (G \times Y') \cap \mathbf{B}$, for $|Y'| \geq \frac{c^2}{4} N$ (as $\xi \leq \frac{1}{2}$).\\

\noindent\textbf{Second application of the Bogolyubov argument.} Finally, convolve in $y$ direction once again, setting
\[D = \cup_x\hspace{2pt}\{x\} \times \supp (C_{x \cdot} \conv C_{x \cdot}) \conv (C_{x \cdot} \conv C_{x \cdot}) \subset A^{(3)}.\]

Let $\chi$ be the indicator function of $Y'$, hence the indicator function of $C$ is $(x,y)\mapsto \mathbf{B}(x,y) \chi(y)$. By Bogolyubov's argument (Lemma~\ref{bogArg}), for every $x$, the set $D_x$ contains a subspace $T_x^\perp$, where $T_x = \Big\{r \colon \Big|\widehat{\mathbf{B}_{x \cdot} \cdot \chi}(r)\Big| \geq C_x^{3/2}\Big\}$ and $C_x = \Big|\widehat{\mathbf{B}_{x \cdot} \cdot \chi}\Big|(0) = \mathbb{E}_y \mathbf{B}_{x \cdot}(y) \chi(y) = N^{-1}\Big|Y' \cap \mathbf{B}_{x \cdot}\Big|$, and $|T_x| \leq C_x^{-2}$.\\

\noindent\textbf{Applying Proposition~\ref{subspaceDenseContainment}.} We now apply Proposition~\ref{subspaceDenseContainment} to $Y'$ as the input set. We obtain a subspace $V$ of codimension at most $20 c^{-2} p^{3k}$ such that $x \in V$ implies that $\Big|Y' \cap \mathbf{B}_{x \cdot}\Big| \geq \frac{c^2}{8p^k} N$. In particular, for all $x \in V$, we have $C_x \geq \frac{c^2}{8p^k}$. It remains to understand the structure of $T_x$ and to relate it to the large spectrum of the indicator function $\chi$ of $Y'$. For this purpose, we recall some basic properties of Fourier transforms.\\

\begin{lemma}\label{subspaceFourier}\begin{itemize}
\item[\textbf{(i)}] For any coset $u_0 + W$ of a subspace, we have
\[\widehat{u_0 + W}(r) = \omega^{-r\cdot u_0}\frac{W^\perp(r)}{|W^\perp|}.\]
\item[\textbf{(ii)}] For any coset $u_0 + W$ of a subspace, we have
\[\Big(\chi \cdot (u_0+ W)\Big)\widehat{\phantom{\Big)}} (r) = |W^\perp|^{-1}\sum_{s \in W^\perp} \omega^{u_0 \cdot s} \widehat{\chi}(r+s).\]
\end{itemize}\end{lemma}

\begin{proof}\textbf{(i)} Suppose that there is some $w_0 \in W$ such that $w_0 \cdot r \not=0$. Let $W'$ be a subspace such that $W' \oplus \langle w_0 \rangle = G$. Then
\[\begin{split}\widehat{u_0+ W}(r) = \mathbb{E}_x (u_0+ W)(x) \omega ^{-rx} & = N^{-1} \sum_x (u_0+ W)(x) \omega ^{-rx}\\
&= N^{-1} \sum_{x \in W'} \sum_{\lambda \in \mathbb{F}_p} (u_0+ W)(x + \lambda w_0) \omega ^{-r\cdot (x + \lambda w_0)}.\end{split}\]
Note that for each $x$, either $(u_0+ W)(x + \lambda w_0) = 0$ for all $\lambda$, or $(u_0+ W)(x + \lambda w_0) = 1$ for all $\lambda$, and in either case $\sum_{\lambda \in \mathbb{F}_p} (u_0+ W)(x + \lambda w_0) \omega ^{-r\cdot (x + \lambda w_0)} = 0$, proving that $\widehat{u_0+ W}(r) = 0$ if $r \notin W^\perp$.\\
\indent On the other hand, if $r \in W^\perp$, we then have $\omega ^{-rx} = \omega^{-r\cdot u_0}$ for all $x \in u_0 + W$, therefore 
\[\widehat{u_0+ W}(r) = N^{-1}\sum_x (u_0+ W)(x) \omega ^{-rx} = N^{-1}\sum_{x \in u_0+ W} \omega ^{-r \cdot u_0} = \frac{|W|}{N} \omega^{-r \cdot u_0} = \omega^{-r\cdot u_0}\frac{W^\perp(r)}{|W^\perp|},\]
as desired.\\
\noindent \textbf{(ii)} We have
\[\begin{split}\Big(\chi \cdot (u_0+ W)\Big)\widehat{\phantom{\Big)}} (r) &= \mathbb{E}_{x,y} \chi(x) (u_0 + W)(y) \omega^{-r \cdot x} \sum_s \omega^{-s \cdot (y-x)}\\
&= \sum_s \Big(\mathbb{E}_x \chi(x) \omega^{-(r-s)\cdot x}\Big)\Big( \mathbb{E}_y (u_0 + W)(y) \omega^{-s\cdot y}\Big) \\
&= \sum_s \widehat{\chi}(r-s) \widehat{u_0 + W}(s)\\
&= |W^\perp|^{-1}\sum_{s \in W^\perp} \omega^{-u_0 \cdot s} \widehat{\chi}(r-s).\end{split}\]\end{proof}

\noindent Recall that for all $x$, the set $D_x$ contains $T_x^\perp$, where $T_x = \Big\{r \colon \Big|\widehat{\mathbf{B}_{x \cdot} \cdot \chi}(r)\Big| \geq C_x^{3/2}\Big\}$ and $C_x = \Big|\widehat{\mathbf{B}_{x \cdot} \cdot \chi}\Big|(0) = \mathbb{E}_y \mathbf{B}_{x \cdot}(y) \chi(y) = N^{-1}\Big|Y' \cap \mathbf{B}_{x \cdot}\Big|$, and $|T_x| \leq C_x^{-2}$. For $x \in V$, we have $C_x \geq \frac{c^2}{8p^k}$, so $|T_x| \leq 64c^{-4}p^{2k}.$ Note that $\mathbf{B}_{x \cdot} \cap Y'$ is relatively large, and
\[\begin{split}\mathbf{B}_{x \cdot} &= \Big\{y \colon (\forall i\in[k])x \cdot \alpha_i(y)= 0\Big\}\\ 
&= \cap_{i \in [k]} \Big\{y \colon x \cdot \alpha_i(y) = 0\Big\}\\
&= \cap_{i \in [k]} \Big\{y \colon x \cdot (\alpha^L_i(y) + \alpha_i(0)) = 0\Big\}\\
&= \cap_{i \in [k]} \Big\{y \colon \alpha^{LT}_i(x) \cdot y = -\alpha_i(0) \cdot x\Big\},\end{split}\]
thus, for $x \in V$, the set $\mathbf{B}_{x \cdot}$ is a coset of $\{\alpha^{LT}_i(x)\colon i\in[k]\}^\perp$. By Lemma~\ref{subspaceFourier}, we have 
\[\Big|\widehat{\mathbf{B}_{x \cdot} \cdot \chi}(r)\Big| \leq \mathbb{E}_{s \in (\{\alpha^{LT}_i(x)\colon i\in[k]\}^\perp)^\perp} |\widehat{\chi}(r + s)|.\]
In particular, if $T = \Big\{r\colon \Big|\widehat{\chi}(r)\Big| \geq \frac{c^3}{24p^{3k/2}}\Big\}$, then we have that
\[\text{if }r \notin T+ \{\alpha^{LT}_i(x)\colon i\in[k]\},\text{ then }\Big|\widehat{\mathbf{B}_{x \cdot} \cdot \chi}(r)\Big| < \frac{c^3}{24p^{3k/2}} \leq \Big|\widehat{\mathbf{B}_{x \cdot} \cdot \chi}(0)\Big|^{3/2}.\]
From this, we conclude that $T_x \subset T + \{\alpha^{LT}_i(x)\colon i\in[k]\}$, for all $x \in V$, and therefore 
$D$ contains $V \times T^\perp) \cap \mathbf{B^L}$, where $\mathbf{B^L}$ is the bilinear Bohr variety defined by maps $(x,y) \mapsto x \cdot \alpha^L_i(y)$, $\operatorname{codim} V  \leq p^k(k + 4c^{-1}p^{2k})$, $|T| \leq 600c^{-6}p^{3k}$ and $k \leq \exp(\pol(\log c^{-1}))$. This concludes the proof of the theorem.\end{proof}

\iftrue
\else

\boldSection{Low-rank Trilinear Forms}

\begin{proof}[Proof of Theorem~\ref{trilinearTheorem}.] Let $A \subset U \times V$ be defined by
\[A = \{(u,v) \in U \times V \colon T(u)(v) = 0\}.\]
Observe that for each $v \in V$, the set $A_{\cdot v}$ is a subspace, and so is $A_{u \cdot}$ for every $u \in U$. From the assumptions on the ranks of maps, using the rank-nullity theorem, we have $|A| \geq c p^{-k} |U||V|$. Apply Corollary~\ref{bisubspace} to $A$ to find $U' \leq U, V' \leq V$ and a bilinear Bohr variety $\mathbf{B^L}$ defined by bilinear maps $\beta_1, \dots, \beta_m$ such that $(U' \times V') \cap \mathbf{B} \subset A$, with
\[\codim_U U', \codim_V V', m \leq \exp(\exp(\pol(\log c^{-1}, k , \log p))).\]
Therefore, whenever $u \in U', v \in V'$ and $\beta_i(u,v) = 0$ holds for all $i \in [m]$, then we have $T(u)(v) = 0$. From this we deduce that
\[\begin{split}&\text{for each }u \in U',\text{ we have }\dim \{T(u)(v) \colon v \in V'\} \leq m,\text{and}\\
&\text{for each }v \in V',\text{ we have }\dim \{T(u)(v) \colon u \in U'\} \leq m.\end{split}\]

To finish the proof, we recall and prove Theorem~\ref{dichotomy}.

\begin{proposition}[Theorem~\ref{dichotomy}, Ramsey-like theorem for bilinear maps.] For any $a,b \in \mathbb{N}$, there is $K$ (we may take $K = (2p)^{\frac{1}{2}(a+1)^2 b}$) such that the following holds. If we are given $\mathbb{F}_p$-vector spaces $X, Y, Z$ and a bilinear map $\beta \colon X \times Y \to Z$ such that $\operatorname{im} \beta$ contains at least $K$ linearly independent vectors, then there is $x \in X$ such that $\operatorname{rank} \beta_{x \cdot} \geq a$ or there is $y \in Y$ such that $\operatorname{rank} \beta_{\cdot y} \geq b$.\end{proposition}

\begin{proof} We prove the claim by induction on $a$. For the basis of induction, note that for any $b \in \mathbb{N}$, if $\operatorname{im} \beta$ has at least one non-zero vector, we may simply pick $(x,y)$ such that $\beta(x,y) \not= 0$ and then $\operatorname{rank} \beta_{x \cdot} \geq 1$, so the propostion holds (with $K(1,b) = 1$ instead of $(2p)^{2b}$).\\

Assume that the claim holds for $a-1, b$. Let $K' = (2p)^{\frac{1}{2}a^2 b}$. Since $\operatorname{im} \beta$ contains at least $K'$ linearly independent vectors, we can apply the induction hypothesis, and we are either done, or, there are $x^{(1)} \in X, y^{(1)}_1, \dots, y^{(1)}_{a-1} \in Y$ such that $\beta(x^{(1)}, y^{(1)}_1), \dots, \beta(x^{(1)}, y^{(1)}_{a-1})$ are linearly independent. We have $K$ linearly independent vectors $\beta(u_1, v_1), \dots, \beta(u_k, v_k)$, where $u_i \in X, v_i \in Y$ and w.l.o.g. $\beta(x^{(1)}, y^{(1)}_1), \dots, \beta(x^{(1)}, y^{(1)}_{a-1})$ are among these.\\
\indent Notice that the number of $i$ such that $u_i \in \langle x^{(1)} \rangle$ is at most $(a-1)(p-1)$, otherwise by pigeonhole principle, some $\beta_{u_i \cdot}$ has rank at least $a$. Now pick a subspace $X_1 \leq X$ of codimension 1, by taking a vector $w \in X$ uniformly at random among vectors such that $w \cdot x^{(1)} = 1$, and define $X_1 = \langle w \rangle ^\perp$. Given any $u \in X$, we have
$$\mathbb{P}(u \in X_1) = 
\begin{cases} 
      1 & u = 0\\
      0 & u \in \langle x^{(1)}\rangle \setminus \{0\}\\
      p^{-1} & \text{otherwise.}
\end{cases}$$
In particular, the expected number of $i$ such that $(u_i, v_i) \in X_1 \times Y$ is at least $(K - (a-1)(p-1))/p \geq K/(2p)$.\\
\indent Hence, we now have $X_1$ of codimension 1 in $X$, such that, $\beta \colon X_1 \times Y \to Z$ has at least $K_1 = (2p)^{\frac{1}{2}(a+1)^2 b - 1}$ linearly independent vectors in its image. In fact, we also have that these vectors are linearly independent from $\beta(x^{(1)}, y^{(1)}_1), \dots, \beta(x^{(1)}, y^{(1)}_{a-1})$. Therefore, if we define
\[Z_1 = \langle \beta(x^{(1)}, y^{(1)}_1), \dots, \beta(x^{(1)}, y^{(1)}_{a-1})\rangle,\]
and denote by $\pi_1 \colon Z \to Z/Z_1$ the projection map, then $\pi_1 \circ \beta \colon X_1 \times Y \to Z/Z_1$ has at least $K_1$ linearly independent vectors in its image.\\
\indent We can apply the same argument again, to obtain $x^{(2)} \in X_1,$ $y^{(2)}_1, \dots,$ $y^{(2)}_{a-1} \in Y$ such that $\beta(x^{(2)}, y^{(2)}_1), \dots,$ $\beta(x^{(2)}, y^{(2)}_{a-1})$ are linearly independent in $Z/Z_1$. Returning to $Z$, this means that in fact the $2(a-1)$ vectors
\[\beta(x^{(1)}, y^{(1)}_1), \dots, \beta(x^{(1)}, y^{(1)}_{a-1}), \beta(x^{(2)}, y^{(2)}_1), \dots, \beta(x^{(2)}, y^{(2)}_{a-1})\]
are linearly independent.\\

We iterate this argument. Because we start with $K = (2p)^{\frac{1}{2}(a+1)^2b}$ and each time we pass to a smaller subspace, we lose a factor of at most $2p$, we can apply this argument $ab$ times in total (as $\frac{1}{2}(a+1)^2b > ab + \frac{1}{2}a^2b$), to obtain linearly independent $x^{(1)}, \dots, x^{(ab)} \in X$, and elements $y^{(i)}_j \in Y$ for $i \in [ab], j \in [a-1]$, such that the vectors $\beta(x^{(i)}, y^{(i)}_j)$ are all linearly independent. Observe that, unless we are either done, we have that for all distinct $i, i' \in [ab]$, $\operatorname{im} \beta_{x^{(i)} \cdot} \cap \operatorname{im} \beta_{x^{(i')} \cdot} = \{0\}$ holds.\\
\indent Now, consider $\beta(x^{(i)}, y^{(1)}_1), \dots, \beta(x^{(i)}, y^{(1)}_{a-1}) \in \operatorname{im} \beta_{x^{(i)} \cdot}$. If it turns out that these are all zero, looking at $\operatorname{im} \beta_{x^{(i)} + x^{(1)} \cdot}$, we see that it contains $\beta(x^{(1)}, y^{(1)}_j)$ for all $j \in [a-1]$, and $\beta(x^{(1)}, y^{(i)}_j) + \beta(x^{(i)}, y^{(i)}_j)$, so we get at least $2(a-1) \geq a$ linearly independent vectors (since $\beta(x^{(1)}, y^{(i)}_j) \in \operatorname{span} \Big\{\beta(x^{(1)}, y^{(1)}_l) \colon l \in [a-1]\Big\}$ for each $j$), as desired.\\
\indent Otherwise, we get for each $i \in [ab]$ a non-zero value $\beta(x^{(i)}, z_i) \in \operatorname{im} \beta_{x^{(i)} \cdot}$, with $z_i \in \{y^{(1)}_j\colon j \in [a-1]\}$. Hence, these $ab$ values are linearly independent vectors and belong to $\Big(\cup_{i \in [a-1]} \operatorname{im} \beta_{\cdot y^{(1)}_j}\Big)$, so by pigeonhole principle, there is is $y^{(1)}_j$ such that $\operatorname{rk} \beta_{\cdot y^{(1)}_j} \geq b$, as desired.\end{proof}

Applying this proposition to the bilinear map $\restr{T}{U' \times V'}$, we have that $E = \operatorname{span} \{T(u)(v) \colon u \in U', v \in V'\}$ has dimension bounded above by $2^{O_p(m^3)}$. This completes the proof.\end{proof}

\fi

\end{document}